\documentclass[12pt]{amsart}
\usepackage{amsmath,amssymb,amscd,amsxtra, amsfonts, amsthm, amsbsy,wasysym,graphics, graphicx}
\usepackage{epstopdf}
\usepackage{algpseudocode}
\usepackage{algorithm}
\usepackage{latexsym}
\usepackage{float}
\usepackage{color}

\usepackage{hyperref}
\usepackage{subfigure}
\usepackage{multirow}

\newtheorem{thm}{Theorem}
\newtheorem{lm}{Lemma}

\newtheorem{Proposition}{Proposition}
\newtheorem{remark}{Remark}
\usepackage{color}

\usepackage{etoolbox}

\title{SINGULAR VECTOR PERTURBATION UNDER GAUSSIAN NOISE}

\author{Rongrong Wang}       

\begin{document}

\maketitle

\begin{abstract}
We perform a non-asymptotic analysis on the singular vector distribution under Gaussian noise. In particular, we provide sufficient conditions on a matrix for its first few singular vectors to have near normal distribution. Our result can be used to facilitate the error analysis in linear dimension reduction.
\end{abstract}

\pagestyle{myheadings}
\thispagestyle{plain}
\markboth{TEX PRODUCTION AND V. A. U. THORS}{SIAM MACRO EXAMPLES}

\section{INTRODUCTION}
Singular value decomposition (SVD) is one of the most important linear Dimension Reduction (DR) techniques, whereas its stability under many important types of noises is not fully studied. In particular,  suppose the random noise has a Gaussian distribution. We are interested in how it affects the low dimension embedding of the data and conversely, how does the dimension reduction change the shape of this noise.

Stability analysis on singular vector and eigenvector subspaces under general perturbations are performed in \cite{Chan,Dopico,Weyl}. In these papers, uniform upper bounds on the rotations of the singular or eigen subspaces are established. Despite the tightness of these results, there is still much room for improvements if we are restricted to special types of noise.

In practice, additive white Gaussian is a common assumption on the perturbation. For analytical convenience, we would like Gaussian distribution of error to be preserved during various dimension reduction processes, but it is not always true.  In SVD, for example,  the lower dimensional data  are the first few singular vectors of data matrices. Then the fact that norms of singular vectors are bounded by one prevents them from being Gaussian. 

As is pointed out in \cite{Bura} in an asymptotic way, that given any low rank data matrix, as the perturbation converges to zero, a properly scaled version of its singular vectors corresponding to the nonzero singular values will converge in distribution to a multivariate normal matrix. This result is closely related and can actually be derived from the well established singular subspace perturbation theory, where expressions of first order (\cite{Li}) and second order (\cite{Va}) approximations to the singular subspaces are explicitly known. 

In this paper, we go further in this direction and provide a non-asymptotic bound on the rate of the convergence. In applications, our result could be used in two ways. For fixed data matrices and noise levels, it can be used to predict whether or not one can safely assume that the dimension reduction does not change the Gaussian shape of the noise. In addition, for a given noise level, it shows how many samples are needed for the Gaussian assumption to be credible.

We now state the mathematical description of our problem. \\  \\
\textbf{Problem}:  \textit{Suppose $Y$ is an $n\times n$ real-valued data matrix, and $\widetilde{Y}=Y+\epsilon W$ is its noisy version, where the entries of $W$ are i.i.d. $N(0,\frac{1}{n})$ and $\epsilon>0$ is some small constant. Suppose for a given $k<n$, $Y$ and $\widetilde{Y}$ have the following conformally partitioned SVDs 
\begin{equation}\label{eq:Y}
 Y=(U_{1},U_{2})\left(\begin{array}{cc}
\Sigma_{1} & 0 \\
0 & 0 \\
\end{array} \right)(V_{1},V_{2})^T,
\end{equation}
\begin{equation}\label{eq:tilde}
\widetilde{Y}=(\widetilde{U}_{1},\widetilde{U}_{2})\left(\begin{array}{cc}
\widetilde{\Sigma}_{1} & 0 \\
0 & \widetilde{\Sigma}_{2} \\
\end{array} \right)(\widetilde{V}_{1},\widetilde{V}_{2})^T,
\end{equation}
where $\Sigma_1$ and $\widetilde{\Sigma}_1$ are $k\times k$ matrices, and all singular values are in descending order. We seek the answer to the following question. For a fixed data matrix $Y$, what values of $\epsilon$ can guarantee the existence of a random unitary matrix (rotation) $M$, such that the random quantity $U_1-\widetilde{U}_1M$ (or $V_1-\widetilde{V}_1M$ if we want to study the right singular vectors) has a distribution that is very similar to Gaussian?}

In other words, we want to know for which $\epsilon$, the perturbed $\widetilde{U}_1$ is approximately a rotated version of the original $U_1$ plus a Gaussian matrix. Rotations of singular vectors can occur when $Y$ has two singular values that are very close to each other, as illustrated in the following example from \cite{Vu}. Let
\[
Y=\left(\begin{matrix}1& 0 &0 \\ 0 &1 &0\\ 0 & 0& \frac{1}{2}\end{matrix}\right), \ \ \ \widetilde{Y}=\left(\begin{matrix}1&\epsilon &0\\ \epsilon&1&0\\0&0&\frac{1}{2}\end{matrix}\right).
\]
Then for all $0<\epsilon< 0.5$,
\[U_1=\left(\begin{matrix}1& 0 \\ 0 &1 \\ 0 &0\end{matrix}\right),\]
is one set of eigenvectors of the two dimensional invariant eigensubspace of $Y$, and 
\[ \widetilde{U}_1=\left(\begin{matrix}1/\sqrt{2}& 1/\sqrt{2} \\ 1/\sqrt{2} &-1/\sqrt{2}\\0&0 \end{matrix}\right).\]
are eigenvectors corresponding to the largest two eigenvalues of $\widetilde{Y}$. 
Therefore, as $\epsilon\rightarrow 0$,  $\widetilde{U}_1$ approaches $U_1M$,
\[ M=\left(\begin{matrix}1/\sqrt{2}& 1/\sqrt{2} \\ 1/\sqrt{2} &-1/\sqrt{2} \end{matrix}\right),
\]
instead of $U_1$.

Let us briefly introduce the history of the singular subspace perturbation theory, on which our analysis will be based.

The first important result about stability of Hermitian eigenvector problems was given in 1970 by Davis and Kahan \cite{Chan}. They proved the sin$\Theta$ theorem, which characterizes the rotation of eigensubspaces of hermitian matrices caused by deterministic perturbations.
In particular, they showed that an eigen-subspace is stable as long as its corresponding eigenvalues and the rest of the spectrum have a large enough gap. Wedin \cite{Wedin} extended this result to singular vector subspaces of non-hermitian matrices. Dopico \cite{Dopico} showed that the rotations of the left and right singular vectors caused by one perturbation should be similar to each other. Van Vu \cite{Vu} considered random perturbations (i.i.d. Bernoulli distribution) and derived an upper bound on the angles of the singular-subspace rotations that is smaller than those in \cite{Wedin,Dopico} but only holds with large probability. We refer the readers to \cite{stewart} for more details of the general stability analysis of SVD.

This note is organized as follows. In Section 2, we introduce notations and some existing results that will be used. In Section 3, we state several lemmas followed by the main theorem. Section 4 contains the proof of the main theorem. In Section 5, we apply our result to an audio signal classification problem, which is also the motivation of this note.

\section{PRELIMINARY RESULTS}
The main result of this note and its proof are both stated for square data matrices, merely for simplicity. All the results could be easily generalized to rectangular data matrices by first padding zeros to form square matrices and then go through the same procedure. Therefore, unless otherwise stated, we assume $Y$ is square with dimensionality $n\geq 2$ and rank $k<n$. The variables $Y$, $\widetilde{Y}$, $\Sigma_1$, $\widetilde{\Sigma}_i$, $U_i$, $\widetilde{U}_i$, $V_i$, $\widetilde{V}_i$, $i=1,2$, remain the same as defined in \eqref{eq:Y} and \eqref{eq:tilde}. 

For a given matrix $A$, we use ($A^H$) $A^T$ to denote its (conjugate) transpose, respectively, and $Q_A$ and $R_A$ to denote some (thin) QR decomposition of $A$. The set of singular values of $A$ is denoted by $\sigma(A)$. The minimum and maximum singular values of $A$ are written as $\sigma_{min}(A)$ and $\sigma_{max}(A)$.  We use $W$ to denote the normalized Gaussian matrix whose entries are $N(0,1/n)$. For any matrix $A$, $\|A\|_F$ denotes its Frobenius norm, $\|A\|_2$ its spectral norm, and $\|A\|_{\max}$ its element-wise maximum, i.e $\|A\|_{\max}=\max\limits_{i,j}|A_{i,j}|$.

The following theorem is from (Theorem 2.1, \cite{Dopico}) which characterizes the stability of singular vector subspaces under deterministic noise.
\begin{thm}(\cite{Dopico})\label{thm:sta} Let $\widetilde{Y}$ and $\widehat{Y}$ has conformally partitioned SVDs \eqref{eq:tilde} and \eqref{eq:hat}
\begin{equation}\label{eq:hat}
\widehat{Y}=(\widehat{U}_{1},\widehat{U}_{2})\left(\begin{array}{cc}
\widehat{\Sigma}_{1} & 0 \\
0 & \widehat{\Sigma}_{2} \\
\end{array} \right)(\widehat{V}_{1},\widehat{V}_{2})^T,
\end{equation}
 Let
\begin{equation}\label{eq:delta}\delta=\min\left\{\left(\min\limits_{\widetilde{\mu}\in \sigma(\widetilde{\Sigma}_2), \mu\in \sigma(\widehat{\Sigma}_1)}|\widetilde{\mu}-\mu|\right), \sigma_{\min}(\widetilde{\Sigma}_1)+\sigma_{\min}(\widehat{\Sigma}_1)\right\}.
\end{equation}
If $\delta>0$, then
\begin{equation}\label{eq:sta}
\min\limits_{L \ unitary}\sqrt{\|\widehat{U}_{1}L-\widetilde{U}_{1}\|^{2}_{F}+\|\widehat{V}_{1}L-\widetilde{V}_{1}\|_{F}^{2}} \leq \sqrt 2\frac{\sqrt{\|R\|_{F}^{2}+\|S\|^{2}_{F}}}{\delta},
\end{equation}
where $ R=(\widehat{Y}-\widetilde{Y})\widetilde{V}_{1}$, $ S=(\widehat{Y}-\widetilde{Y})^T\widetilde{U}_{1}$. Moreover, the left hand side of \eqref{eq:sta}  is minimized for $L=Z_1Z_2^{T}$ where $Z_1SZ_2^{T}$ is any SVD of $\widehat{U}_{1}^{T}\widetilde{U}_{1}+\widehat{V}_{1}^{T}\widetilde{V}_{1}$, and the equality can be attained.
\end{thm}

In the proof of the main theorem (Theorem \ref{thm:main}), we will frequently encounter the spectral norm of Gaussian matrices, which is bounded in the following theorem.
\begin{thm}(\cite{Szarek})\label{thm:condition}
Let W be an $n\times k$ ($k\leq n$) matrix with i.i.d. normal entries with mean 0 and variance $1/n$. Then, its largest and smallest singular values obey:
\[ P(\sigma_{1}(W)>1+\sqrt{k/n}+r)\leq e^{-nr^{2}/2},
\]
\[ P(\sigma_{k}(W)<1-\sqrt{k/n}-r)\leq e^{-nr^{2}/2},
\]
for any $r>0$.
\end{thm} 

The estimates in Theorem \ref{thm:condition} uses the joint density of singular value distribution of Gaussian matrix, whose explicit expression involves hypergeometric functions of matrix argument (see e.g., \cite{Muirhead} for details).

When applying Theorem \ref{thm:sta}, we need to use the following result of Weyl \cite{Weyl} to estimate the $\delta$ in \eqref{eq:sta}.
\begin{thm}\label{thm:eigs}
Suppose $A$ and $\widetilde{A}$ are both $n\times n$ Hermitian and that their eigenvalues are ordered decreasingly. Then
\[
\max\limits_{i=1,..,n}\{\lambda_i-\widetilde{\lambda}_i\} \leq \|A-\widetilde{A}\|_2.
\]
\end{thm}
Last we need a perturbation result on the Cholesky decomposition (Theorem 1.1 of \cite{sun}).
\begin{thm}(\cite{sun})\label{thm:QR}
Let $A$ be an $n\times n$ positive definite matrix and $A=LL^H$ its Cholesky factorization. If $K$ is an $n\times n$ Hermitian matrix satisfying $\|A^{-1}\|_2\|K\|_F<1/2$, then there is a unique Cholesky factorization
\[A+K=(L+G)(L+G)^H\]
and
\[
\frac{\|G\|_F}{\|L\|_2}\leq \sqrt 2 \frac{\kappa_2(A)\|K\|_F/\|A\|_2}{1+\sqrt{(1-2\kappa_2(A)\|K\|_F/\|A\|_2)}},
\]
where $\kappa_2(A)=\|A\|_2\|A^{-1}\|_2$.
\end{thm}
\section{MAIN RESULT}
We are now ready to state our main result.
\begin{thm}\label{thm:main}
Let $Y$ and $\widetilde{Y}$ and their SVDs be defined as in \eqref{eq:Y} and \eqref{eq:tilde}. Assume that $\epsilon$ and $\frac{1}{n}$ are small enough such that the following defined $E_1$ and $\delta_1$ satisfy $E_1<1/(2\sqrt k)$ and $\delta_1>0$, then for any $0<\beta<1/2$ and $\gamma>0$, with probability (with respect to the random Gaussian noise W) exceeding $1-3e^{-(n-k)^\beta+\ln(k(k+n))}-5e^{-(n-k)\gamma^2/2}$, there exists a unitary $k \times k$ matrix $M$, such that
\begin{equation}\label{eq:thm}
\|\widetilde{U}_1M-(U_1+\epsilon N)\|_{\max}\leq \frac{4\sqrt{k}E_3}{\delta_1}+2k \frac{E_1}{1+\sqrt{1-2\sqrt k E_1}}(\|U_1\|_{\max}+E_2)+E_4.
\end{equation}
where $N$ is a Gaussian matrix (matrix with all entries being Gaussian) defined by $N= U_2U_2^TWV_1\Sigma_1^{-1}$, and where
\begin{eqnarray*}
E_1(\epsilon,\Sigma,k,n,\gamma)&=&\epsilon^2\|\Sigma_1^{-1}\|_2^2\alpha_1^2+2\epsilon^3\|\Sigma_1^{-1}\|_2^3\alpha_1^2\alpha_2+\epsilon^4\|\Sigma_1^{-1}\|^4_2\alpha_1^2\alpha_2^2,\\
E_2(\epsilon,\Sigma,k,n,\gamma)&=&\epsilon \alpha_1\|\Sigma_1^{-1}\|_2+\epsilon^2\|\Sigma_1^{-1}\|_2^2\alpha_1\alpha_3, \\
 E_3(\epsilon,\Sigma,k,n,\gamma) &= &\epsilon^3 \|\Sigma^{-1}\|_2^2(\alpha_1^2\alpha_3+2\alpha_1^2\alpha_2+2\alpha_1\alpha_2\alpha_3)+\epsilon^4 \|\Sigma^{-1}\|_2^3(\alpha_1^2\alpha_2^2+2\alpha_1^2\alpha_2\alpha_3)\notag \\ &+&\epsilon^5 \|\Sigma^{-1}\|_2^4(\alpha_1^2\alpha_2^2\alpha_3), \\
E_4(\epsilon,\Sigma,k,n,\gamma)&=& 2\epsilon^2(1+k)(n-k)^{-\frac{1}{2}+\beta}\|\Sigma_1^{-1}\|_2^2 ,\\
 \delta_1&=&  \frac{2}{3}\left(\sigma_{\min}(\Sigma_1)-\epsilon(\alpha_2-\alpha_1)-2\epsilon^2\|\Sigma_1^{-1}\|_2\alpha_1\right)-\sqrt{6\epsilon \|\Sigma_1\|_2+9\epsilon^2},
\end{eqnarray*}
and $\alpha_1=1+\gamma+\sqrt{\frac{k}{n}}$, $\alpha_2=3+2\gamma+2\sqrt{\frac{k}{n}}$, $\alpha_3=2+\gamma$.
\end{thm}
Note that the order of $E_1$-$E_4$ in terms of $\epsilon$ and $n$ are: $E_1 =O (\epsilon^2)$,  $E_2=O(\epsilon)$, $E_3 =O( \epsilon^3)$,  $E_4 =O(\epsilon^2 n^{-\frac{1}{2}+\beta})$ and $\delta_1=O(1)$. The derivation of these terms can be found in \eqref{eq:E1},\eqref{eq:E2},\eqref{eq:E3},\eqref{eq:E4} and \eqref{eq:delta1} in the next section.
\begin{remark} Here we use the element-wise maximum norm instead of the Frobenius norm to characterize the error because from the dimension reduction point of view, we care about all the points in the data matrix instead of the matrix as a whole.
\end{remark}
\begin{remark}
When assuming $rank(Y)=k$, we have allowed infinitely  many choices of $U_2$ in the SVD of $Y$ but all of them span the same singular subspace of $Y$. It is easy to verify that the result and all the terms defined in Theorem \ref{thm:main} only depend on this singular subspace instead of any particular choice of $U_2$.
\end{remark}
\begin{remark}\label{rmk:2} We observe that the difference $\widetilde{U}_1M-U_1$ is approximately  Gaussian only when the Gaussian term $\epsilon N$ is the leading term in the error. The magnitude of each element of $\epsilon N $ has asymptotic order of $O(\epsilon /\sqrt{n})$ as $\epsilon$ and $\frac{1}{n}$ going to 0, while the order of the leading terms on the right hand side is either $\epsilon^2\|U_1\|_{\max}$, $\epsilon^2n^{-\frac{1}{2}+\beta}$, or $\epsilon^3$. Therefore, roughly speaking, to ensure the Gaussian term has the lowest order, we need the following condition for the $(\epsilon, n)$ pair:
\begin{equation}\label{eq:epsilon} \epsilon =o\left( \min\left\{n^{-\beta}, n^{-1/4}, \frac{1}{\|U_1\|_{\max}n^{\frac{1}{2}}}\right\}\right).\end{equation}
In fact, Remark \ref{rmk:4} indicates that this estimate can be further improved to exclude the $n^{-1/4}$ term from the right-hand side of \eqref{eq:epsilon}. Since it involves a much more complicated argument with little innovation, it is beyond the scope of this note. Even from \eqref{eq:epsilon}, we are able to tell that the Gaussianity of the lower dimensional data depends on how the original eigenvectors $U_1$ are spread out. If the entires of $U_1$ have constant magnitude as that in the example in section 5, then $\epsilon=o(n^{-1/4})$ is sufficient.
\end{remark}
\begin{remark}
A first application to our result, which is also the motivation to this paper, is an audio signal classification problem outlined in Section 5. There, the perturbation problem is $\widetilde{Y}=Y+\frac{1}{\sqrt n} W$ ($W$ is a standard Gaussian matrix). Instead of being a free variable, the noise level is now a function of the size $n$ of $Y$.  As we can no longer fix a matrix and let the noise level goes to 0, classical results (\cite{Li, Va}) are not directly applicable to this problem. 
\end{remark} 

Before going into the proof of the main theorem, we first establish three useful lemmas. The first one is an elementary observation, so we omit its proof.
\begin{lm}\label{lm:1}
If $A=\left(\begin{matrix} 0 & A_{12} \\ A_{21} & 0\end{matrix}\right )$, then $\|A\|_{2}\leq  \max\left\{\|A_{12}\|_2,\|A_{21}\|_2\right\}.$
\end{lm}


\begin{lm}\label{lm:2}
Let $\mathcal{F}$ be the distribution of the product of two independent $N(0,1)$ random variables. Let $X_1$, $X_2$,..., $X_n$ be i.i.d. random variables drawn from $\mathcal{F}$. If $n\geq 2$, we have
\[
P(\overline{X}>n^{-\frac{1}{2}+\beta})\leq 2e^{-n^{\beta}}, \ \ \ \ \ \ \ \text{ for any    } \beta \in \mathbb{R},
\]
where $\overline{X}=\left(\sum\limits_{i=1}^{n}X_i\right)/n$.
\end{lm}
\begin{proof}
If $X_i \sim \mathcal{F}$, one can verify that for all $0\leq \theta<1$,
\[
Ee^{\theta X_i}=\sqrt{\frac{1}{1-\theta^2}}.
\]
We apply Markov's inequality (see eg. \cite{Stein}) to have:
\[
P(\overline{X}\geq n^{-\frac{1}{2}+\beta})=P(e^{t\overline{X}}\geq e^{tn^{-\frac{1}{2}+\beta}})\leq \frac{Ee^{t\overline{X}}}{e^{tn^{-\frac{1}{2}+\beta}}}, \ \ \ \ \text{for any } t \in \mathbb{R}^+.
\]
Letting $t=n^{1/2}$ in the above formula, we get
\begin{eqnarray*}
&&P(\overline{X}\geq n^{-\frac{1}{2}+\beta}) \leq \frac{Ee^{n^{1/2}\overline{X}}}{e^{n^{\beta}}}
= e^{-n^{\beta}}\prod\limits_{i=1}^{n}E(e^{n^{1/2}X_i}) \\
&=&  e^{-n^{\beta}}(1-\frac{1}{n})^{-n/2}
\leq  e^{-n^{\beta}}(1-\frac{1}{2})^{-2/2}
=  2e^{-n^{\beta}}.
\end{eqnarray*}
\end{proof}
\begin{lm}\label{lm:3}
Let $T$ be an $n\times n$ Gaussian matrix whose elements are i.i.d. $N(0,\frac{1}{n})$. Let $T$ be written as
\[
T=\left(\begin{matrix}T_{11} & T_{12} \\T_{21} & T_{22} \end{matrix}\right),\]
with $T_{11}$ a $k\times k$ matrix. Then, with probability exceeding $1-4e^{-(n-k)\gamma^2/2}$,
\[
\|T_{11}\|_2\leq \gamma+2\sqrt{\frac{k}{n}}, \ \ \
\|T_{21}\|_2,\ \|T_{12}\|_2\leq 1+\gamma+\sqrt{\frac{k}{n}},\ \ \
\|T\|_2,\ \|T_{22}\|_2\leq 2+\gamma.
\]
\end{lm}
\begin{proof}
Applying Theorem \ref{thm:condition} to the matrix $\sqrt{\frac{n}{k}}T_{11}$, we get
\[
P\left(\sqrt{\frac{n}{k}}\|T_{11}\|>2+\widetilde{\gamma}\right)\leq e^{-k\widetilde{\gamma}^2/2}.
\]
Let $\gamma=\sqrt{\frac{n}{k}}\widetilde{\gamma}$ and substitute it into the above equation to obtain
\[
P\left(\|T_{11}\|>2\sqrt{\frac{k}{n}}+\gamma\right)\leq e^{-n\gamma^2/2}.
\]
The other four inequalities can be derived similarly.
\end{proof}
\section{PROOF OF THE MAIN THEOREM}
\begin{proof}
 Define the random matrix $C$ as follows,
\begin{equation}\label{eq:ori}
U^{T}\widetilde{Y}V=\Sigma+\epsilon U^{T}WV = \Sigma +\epsilon C,
\end{equation}
where $\Sigma=\left(\begin{matrix}\Sigma_1 & 0 \\ 0 & 0\end{matrix}\right)$ and $C = \left( \begin{array}{ccc} C_{11} & C_{12} \\ C_{21} & C_{22} \end{array} \right)
 = \left( \begin{array}{ccc} U_{1}^{T}WV_{1} & U_{1}^{T}WV_{2} \\ U_{2}^{T}WV_{1} & U_{2}^{T}WV_{2} \end{array} \right).$ Due to the invariant property of Gaussian matrix, $C$ has the same distribution as $W$. We want to further diagonalize $C$ by using the following two matrices,
\[
P = I + \epsilon \left( \begin{array}{ccc} 0 & -\Sigma^{-1}_{1}C^{T}_{21} \\C_{21}\Sigma^{-1}_{1}  & 0 \end{array} \right)+\epsilon^{2}\left( \begin{array}{ccc} 0 & B^{T} \\ -B & 0 \end{array} \right),
\]
with $B=-C_{22}C_{12}^{T}\Sigma^{-2}_1+C_{21}\Sigma^{-1}_{1}C_{11}\Sigma_{1}^{-1}$, and
\begin{equation}\label{eq:OP}
O = I + \epsilon \left( \begin{array}{ccc} 0 & -\Sigma^{-1}_{1}C_{12} \\C_{12}^{T}\Sigma^{-1}_{1}  & 0 \end{array} \right)+\epsilon^{2}\left( \begin{array}{ccc} 0 & D \\ -D^T& 0 \end{array} \right),
\end{equation}
with $D=-\Sigma_1^{-2}C_{21}^TC_{22}+\Sigma_1^{-1}C_{11}\Sigma^{-1}_{1}C_{12}$.  
We multiply the left hand side of \eqref{eq:ori} by $P^T$ on the left, and by $O$ on the right, to obtain,
\begin{eqnarray}\label{eq:diag}
&&P^{T}U^{T}\widetilde{Y}VO =  \left(I + \epsilon \left( \begin{array}{ccc} 0 & -\Sigma^{-1}_{1}C^{T}_{21} \\C_{21}\Sigma^{-1}_{1}  & 0 \end{array} \right)+\epsilon^{2}\left( \begin{array}{ccc} 0 & B^{T} \\ -B & 0 \end{array} \right)\right)^T\notag\\ &&\times \left( \Sigma +\epsilon \left( \begin{array}{ccc} C_{11} & C_{12} \\ C_{21} & C_{22} \end{array} \right)\right)\notag\left( I + \epsilon \left( \begin{array}{ccc} 0 & -\Sigma^{-1}_{1}C_{12} \\C_{12}^{T}\Sigma^{-1}_{1}  & 0 \end{array} \right)+\epsilon^{2}\left( \begin{array}{ccc} 0 & D \\ -D^T& 0 \end{array} \right)\right) \notag\\
&=&\left(\begin{matrix}\Sigma_1+\epsilon C_{11}+\epsilon^2(\Sigma_1^{-1}C_{21}^TC_{21}+C_{12}C_{12}^{T}\Sigma^{-1}_1 )&0 \\0& \epsilon C_{22}-2\epsilon^2C_{21}\Sigma^{-1}_1C_{12}\end{matrix}\right)+E,
\end{eqnarray}
where the matrix $E$ includes all the terms whose order on $\epsilon$ is greater than or equal to 3, i.e., $E$ is of $O(\epsilon^3)$. 

If we write $UP$ as $UP=((UP)_1,(UP)_2)$, and $VO$ as $VO=((VO)_1,(VO)_2)$ with $(UP)_1$ and $(VO)_1$ being $n\times k$, then the submatrices have the following expressions,
\begin{eqnarray}\label{eq:up}
(UP)_1&=&U_1+\epsilon U_2C_{21}\Sigma_1^{-1}-\epsilon^2 U_2 B, \notag\\
(UP)_2&=&U_2-\epsilon U_1\Sigma_1^{-1}C_{21}^T+\epsilon^2U_1B^T,\notag\\
(VO)_1&=&V_1+\epsilon V_2C_{12}^T\Sigma_1^{-1}-\epsilon^2 V_2 D^T,\notag\\
(VO)_2&=&V_2-\epsilon V_1\Sigma_1^{-1}C_{12}^T+\epsilon^2V_1D.
\end{eqnarray}
It can be directly verified that $(UP)_1^T(UP)_2=0$. This implies that $(UP)_1$ and $(UP)_2$ are orthogonal to each other, and the same result holds for $(VO)_1$ and $(VO)_2$. However, none of these four matrices are orthogonal matrices themselves.
We orthogonalize them by multiplying both sides of equation \eqref{eq:diag} by $\left( \begin{matrix}
R_{(UP)_1}^{-T} & 0 \\ 0 & R_{(UP)_2}^{-T}
\end{matrix}\right)$ on the left, and by $\left( \begin{matrix}
R_{(VO)_1}^{-1} & 0 \\ 0 & R_{(VO)_2}^{-1}
\end{matrix}\right)$ on the right, where $Q_{A}$ and $R_{A}$ denote a QR decomposition of a matrix $A$. This way we obtain:
\begin{eqnarray}\label{eq:leftmulti}
&&\left( \begin{matrix}
R_{(UP)_1}^{-T} & 0 \\ 0 & R_{(UP)_2}^{-T}
\end{matrix}\right) P^{T}U^{T}\widetilde{Y}VO\left( \begin{matrix}
R_{(VO)_1}^{-1} & 0 \\ 0 & R_{(VO)_2}^{-1}
\end{matrix}\right)= \left( \begin{matrix}L_1 & 0 \\ 0 & L_2 \end{matrix}\right)+E,
\end{eqnarray}
where
\begin{eqnarray*} L_1&=&R_{(UP)_1}^{-T}(\Sigma_1+\epsilon C_{11}+\epsilon^2(\Sigma_1^{-1}C_{21}^TC_{21}+C_{12}C_{12}^{T}\Sigma_1^{-1})) R_{(VO)_1}^{-1},\\
 L_2&=&R_{(UP)_2}^{-T}( \epsilon C_{22}-\epsilon^2C_{21}\Sigma^{-1}_1C_{12})R_{(VO)_2}^{-1}.\end{eqnarray*}
With a little abuse of notation, we continue to use $E$ to denote the error term in \eqref{eq:leftmulti}, though it was already changed by the left and right multiplications.
We denote by $M_i\widehat{\Sigma}_iM_i'$ the SVD of the $L_i$, with $i=1,2$. In this case, \eqref{eq:leftmulti} becomes
\begin{eqnarray}\label{eq:dec}
\left( \begin{matrix}
R_{(UP)_1}^{-T} & 0 \\ 0 & R_{(UP)_2}^{-T}
\end{matrix}\right) P^{T}U^{T}\widetilde{Y}VO\left( \begin{matrix}
R_{(VO)_1}^{-1} & 0 \\ 0 & R_{(VO)_2}^{-1}
\end{matrix}\right)= \left( \begin{matrix}
M_1\widehat{\Sigma}_1M_1' & 0 \\ 0 & M_2\widehat{\Sigma}_2M_2'
\end{matrix}\right)+E.\end{eqnarray}
On the other hand, the left hand side of \eqref{eq:dec} can be simplified to
\begin{eqnarray} \label{eq:main}
&&\left( \begin{matrix}
R_{(UP)_1}^{-T} & 0 \\ 0 & R_{(UP)_2}^{-T}
\end{matrix}\right) P^{T}U^{T}\widetilde{Y}VO\left( \begin{matrix}
R_{(VO)_1}^{-1} & 0 \\ 0 & R_{(VO)_2}^{-1}
\end{matrix}\right)\notag\\
&=& \left(\begin{matrix} R_{(UP)_1}^{-T}(UP)_1^T \\  R_{(UP)_2}^{-T}(UP)_2^T \end{matrix}\right)\widetilde{Y}\left(\begin{matrix} (VO)_1R_{(VO)_1}^{-1} & (VO)_2 R_{(VO)_2}^{-1} \end{matrix}\right) \notag\\
&=& \left( \begin{matrix}
Q_{(UP)_1}& Q_{(UP)_2}
\end{matrix}\right)^T\widetilde{Y} \left( \begin{matrix}
Q_{(VO)_1}& Q_{(VO)_2}
\end{matrix}\right).
\end{eqnarray}
Combining \eqref{eq:dec} with \eqref{eq:main}, we obtain:
\begin{equation}\label{eq:ort}
\left( \begin{matrix}
Q_{(UP)_1}& Q_{(UP)_2}
\end{matrix}\right)^T\widetilde{Y} \left( \begin{matrix}
Q_{(VO)_1}& Q_{(VO)_2}
\end{matrix}\right)=\left( \begin{matrix}
M_1\widehat{\Sigma}_1M_1' & 0 \\ 0 & M_2\widehat{\Sigma}_2M_2'
\end{matrix}\right)+E.
\end{equation}
Notice that matrices that are left and right to $\widetilde{Y}$ in \eqref{eq:ort} are unitary, because $Q_{(UP)_i}$ $i=1,2$ have the same span as $(UP)_i$, $i=1,2$, respectively, and we know that $(UP)_1$ and $(UP)_2$ are mutually orthogonal. Moving everything but $\widetilde{Y}$ on the left hand side to the right by multiplying the transpose of each matrix, we derive:
\begin{equation}\label{eq:I}
\widetilde{Y}= \left( \begin{matrix}
Q_{(UP)_1}M_1& Q_{(UP)_2}M_2
\end{matrix}\right) \left( \begin{matrix}
\widehat{\Sigma}_1& 0 \\ 0 & \widehat{\Sigma}_2
\end{matrix}\right)\left( \begin{matrix}
Q_{(VO)_1}M_1'& Q_{(VO)_2}M_2'
\end{matrix}\right)^T+E.
\end{equation}
We will show later that $\sigma_{\min}({\widehat{\Sigma}_1})-\|\widetilde{\Sigma}_2)\|_2>0$ when $\epsilon$ is small enough. Then \eqref{eq:I} combined with Theorem \ref{thm:sta} imply that $\widetilde{U}_1$ and $Q_{(UP)_1}M_1$ are the left singular vectors of two very similar matrices (different by the error matrix $E$) and so that they are close. Keeping this useful result in mind, we first turn to look at the big picture.

Our final goal is to approximate by a Gaussian variable the difference between $U_1$ and $\widetilde{U}_1$ up to a rotation: $U_1-\widetilde{U}_1M$ (we will define the unitary matrix M explicitly later), which can be decomposed as:
\begin{equation}\label{eq:decom1}
\widetilde{U}_1M-U_1=(\widetilde{U}_1M-Q_{(UP)_1})+(Q_{(UP)_1}-(UP)_1)+((UP)_1-U_1).
\end{equation}
We insert \eqref{eq:up} into \eqref{eq:decom1} to get
\begin{eqnarray*}
\widetilde{U}_1M-U_1
=\left(\widetilde{U}_1M-Q_{(UP)_1}\right)+(Q_{(UP)_1}-(UP)_1)+(\epsilon U_2C_{21}\Sigma_1^{-1}-\epsilon^2 U_2 B).
\end{eqnarray*}
Here, the first term in the last parentheses is Gaussian,  and we want to prove all other terms are small. For that purpose, we move the Gaussian term to the left and take the component-wise matrix norm on both sides, to have:
\begin{eqnarray}\label{eq:III}
&&\|\widetilde{U}_1M-U_1-\epsilon U_2C_{21}\Sigma_1^{-1}\|_{\max}\notag \\ &\leq&\|Q_{(UP)_1}-(UP)_1\|_{\max}+\|\widetilde{U}_1M-Q_{(UP)_1}\|_{\max}+\epsilon^2\|U_2 B\|_{\max}\\
&=& I+II+III \notag.
\end{eqnarray}
Observe that the left hand side of \eqref{eq:III} is exactly what we want to bound in this theorem.
The rest of the proof is divided into three parts to bound each term $I$, $II$, $III$ on the right hand side. \\
\subsection{ESTIMATING I}
We start with calculating how far away is $UP$ from unitary.
\[
P^TU^TUP=P^TP=I+\epsilon^2\left(\begin{matrix}\Sigma_1^{-1}C_{21}^TC_{21}\Sigma_{1}^{-1} &0\\0& C_{21}\Sigma_1^{-2}C_{21}^T\end{matrix}\right)+O(\epsilon^3).
\]
Hence,
\[(UP)_1^T(UP)_1=I+\epsilon^2(\Sigma_1^{-1}C_{21}^TC_{21}\Sigma_{1}^{-1})+O(\epsilon^3).
\]
It can be directly calculated (see Proposition 1 below) that
\begin{equation}\label{eq:spebd}
\|(UP)_1^T(UP)_1-I\|_2 \leq E_1(\epsilon,\Sigma,k,n,\gamma),
\end{equation}
with probability over $1-4e^{-(n-k)\gamma^2/2}$.
Where
\begin{equation}\label{eq:E1}
E_1=\epsilon^2\|\Sigma_1^{-1}\|_2^2\alpha_1^2+2\epsilon^3\|\Sigma_1^{-1}\|_2^3\alpha_1^2\alpha_2+\epsilon^4\|\Sigma_1^{-1}\|^4_2\alpha_1^2\alpha_2^2.
\end{equation}Thus,
\begin{equation}\label{eq:sigma}
\|(UP)_1^T(UP)_1-I\|_F \leq\sqrt k \|(UP)_1^T(UP)_1-I\|_2 \leq\sqrt{k}E_1(\epsilon,\Sigma,k,n,\gamma).
\end{equation}
Applying Theorem \ref{thm:QR} by assigning $A=I$, $L=I$, $\widetilde {A}=(UP)_1^T(UP)_1=R_{(UP)_1}^TR_{(UP)_1}$, $K=\widetilde{A}-A=(UP)_1^T(UP)_1-I$ and $G=R_{(UP)_1}^T-I$, we obtain
\begin{eqnarray}\label{eq:dist}
\|Q_{(UP)_1}-(UP)_1\|_2&=& \|Q_{(UP)_1}(I-R_{(UP)_1})\|_2= \|I-R_{(UP)_1}\|_2 \notag \\
&\leq& \|I-R_{(UP)_1}\|_F=\|G\|_F\leq \frac{\sqrt 2\|K\|_F}{1+\sqrt{1-2\|K\|_F}} \leq \frac{\sqrt{2k} E_1}{1+\sqrt{1-2\sqrt k E_1}},
\end{eqnarray}
where the last inequality made use of \eqref{eq:sigma}. For fixed $Y$ and $\gamma$, equation \eqref{eq:dist} has implied the necessity of imposing the condition $E_1\leq 1/2\sqrt k$ on $\epsilon$ for the term under the square root to be positive, which we assume to be true from now on. We proceed to calculate
\begin{eqnarray}\label{eq:3}
I&=&\|Q_{(UP)_1}-(UP)_1\|_{\max} \notag\\
&=&\|(UP)_1(R_{(UP)_1}^{-1}-I)\|_{\max} \notag\\
&\leq & \sqrt{k}\|(UP)_1\|_{\max}\cdot\|R_{(UP)_1}^{-1}-I\|_2 \notag\\
&= & \sqrt{k}\|(UP)_1\|_{\max}\cdot\|R_{UP}^{-1}\|_2\cdot\|I-R_{(UP)_1}\|_2 \notag\\
&\leq & \sqrt{k}\|(UP)_1\|_{\max}\cdot(\sigma_{\min}((UP)_1))^{-1}\cdot\|Q_{(UP)_1}-(UP)_1\|_2.
\end{eqnarray}
From \eqref{eq:sigma} and the assumption that $E_1<\frac{1}{2\sqrt{k}}$, we know that $|\sigma_{min}((UP)_1)^2-1|<\frac{1}{2}.$ Therefore,
\begin{equation}\label{eq:inverse} \frac{1}{\sigma_{min}((UP)_1)}<\sqrt{2}.\end{equation}
We insert \eqref{eq:dist} and \eqref{eq:inverse} into \eqref{eq:3}, to arrive at the bound:
\[
I \leq 2k \frac{E_1}{1+\sqrt{1-2\sqrt k E_1}}\|(UP)_1\|_{\max}.
\]
Furthermore, from \eqref{eq:up} and Lemma \ref{lm:3}, it is straightforward to estimate that with probability $1-4e^{-(n-k)\gamma^2/2}$,
\begin{eqnarray}\label{eq:E2}
&&\|(UP)_1\|_{\max}\leq  \|U_1\|_{\max}+\epsilon\|U_2C_{21}\Sigma_1^{-1}\|_2+\epsilon^2\|U_2B\|_2 \notag \\
&& \leq\|U_1\|_{\max}+\epsilon \alpha_1\|\Sigma_1^{-1}\|_2+\epsilon^2\|\Sigma_1^{-1}\|_2^2\alpha_1\alpha_2\equiv \|U_1\|_{\max}+E_2,
\end{eqnarray}
with $\alpha_1=1+\gamma+\sqrt{\frac{k}{n}}$, $\alpha_2=3+2\gamma+2\sqrt{k/n}$ same as those defined in Theorem \ref{thm:main}.
We combine the above two inequalities to get:
\[
I\leq 2k \frac{E_1}{1+\sqrt{1-2\sqrt k E_1}} (\|(U_1\|_{\max}+E_2).
\]

\subsection{ESTIMATING II}
From \eqref{eq:tilde} and \eqref{eq:I}, we know that $(\widetilde{U}_1,\widetilde{U}_2)$ and $(Q_{(UP)_1}M_1,Q_{(UP)_2}M_2)$ are the left singular vectors of $\widetilde{Y}$ and $\widetilde{Y}-E:=\widehat{Y}$, respectively. Thus, we want to use Theorem \ref{thm:sta} to bound $II$. For this purpose, we need to control both $\|\widetilde{\Sigma}_2\|$ and $\sigma_{\min}(\widehat{\Sigma}_1)$.

Let us first estimate $\sigma_{\min}(\widehat{\Sigma}_1)$, which stores the singular values of the matrix $L_1=R_{(UP)_1}^{-T}(\Sigma_1+\epsilon C_{11}+\epsilon^2(\Sigma_1^{-1}C_{21}^TC_{21}+C_{12}C_{12}^{T}\Sigma_1^{-1})) R_{(VO)_1}^{-1}$. From equation \eqref{eq:spebd} and the assumption $E_1<\frac{1}{2\sqrt k}<\frac{1}{2}$, we have
\[
\sigma_{\min}(R^{-1}_{(UP)_1})=\frac{1}{\|(UP)_1\|_2}\geq \sqrt{\frac{2}{3}}.
\]
Similarly, $\sigma_{\min}(R^{-1}_{(VO)_1})\geq \sqrt{2/3}$. Moreover, from Lemma \ref{lm:3}, we know that with probability $1-3e^{-(n-k)\gamma^2/2}$, it hold
\[
\|T_{11}\|_2\leq \gamma+2\sqrt{\frac{k}{n}}, \ \ \ \|T_{12}\|_2, \|T_{21}\|_2\leq 1+\gamma+\sqrt{\frac{k}{n}}.
\]
Combining these facts yields,
\begin{equation}\label{eq:sigmahat}
\sigma_{\min}(\widehat{\Sigma}_1)=\sigma_{\min}(L_1)\geq \frac{2}{3}(\sigma_{\min}(\Sigma_1)-\epsilon(\alpha_2-\alpha_1)-2\epsilon^2\|\Sigma_1^{-1}\|_2\alpha_1^2),
\end{equation}
where $\alpha_1=1+\gamma+\sqrt{k/n}$, and $\alpha_2=3+2\gamma+2\sqrt{k/n}$ as defined in Theorem \ref{thm:main}.

Now let us bound $\|\widetilde{\Sigma}_2\|_2$. Recall that $\widetilde{Y}=Y+\epsilon W$. Then
\[
\widetilde{Y}^T\widetilde{Y}=Y^TY+\epsilon W^T Y + \epsilon Y^TW + \epsilon^2 W^T W \notag.
\]
We apply Theorem \ref{thm:eigs} to $Y^TY$ and $\widetilde{Y}^T\widetilde{Y}$, to get for any $i=1,...,n$,
\begin{eqnarray*}
|\lambda_i^2-\tilde{\lambda}_i^2|&\leq& \|\epsilon W^TY+\epsilon Y^TW+\epsilon^2W^TW \|_2 \notag\\
&\leq&2\epsilon\|W\|_2\|Y\|_2+\epsilon^2\|W\|^2_2 \notag \\
 &\leq&6\epsilon \|\Sigma_1\|_2+9\epsilon^2, \ \ \ \ \ \ \text{with  probability} \ 1-e^{-n/2}.
\end{eqnarray*}
The last inequality made use of Theorem \ref{thm:condition}. \\ Hence for any $\widetilde{\lambda}_i \in \widetilde{\Sigma}_2$, we have
\[
|\widetilde{\lambda}_i^2-0|\leq(6\epsilon \|\Sigma_1\|_2+9\epsilon^2),
\]
and thus
\begin{equation}\label{eq:0}
\|\widetilde{\Sigma}_2\|_2\leq \sqrt{6\epsilon \|\Sigma_1\|_2+9\epsilon^2}.
\end{equation}
It is easy to verify that the $\delta$ defined in \eqref{eq:delta} obeys
\begin{equation}\label{eq:delta1}
\delta\geq \sigma_{\min}(\widehat{\Sigma}_1)-\|\widetilde{\Sigma}_2\|_2\geq \frac{2}{3}(\sigma_{\min}(\Sigma_1)-\epsilon(\alpha_2-\alpha_1)-2\epsilon^2\|\Sigma_1^{-1}\|_2\alpha_1^2)-\sqrt{6\epsilon \|\Sigma_1\|_2+9\epsilon^2}\equiv \delta_1,
\end{equation}
where second inequality follow from \eqref{eq:sigmahat} and \eqref{eq:0}.
 
Whenever $\delta_1>0$, we can apply Theorem \ref{thm:sta} to the two SVDs in \eqref{eq:tilde} and \eqref{eq:I}, to obtain
\begin{equation}\label{eq:lemma1}
\min\limits_{L \ unitary}\|Q_{(UP)_1}M_1L-\widetilde{U}_{1}\|_{F} \leq \sqrt 2 \frac{\sqrt{\|E^T\widetilde{U}_1\|_{F}^{2}+\|E\widetilde{V}_1\|^{2}_{F}}}{\delta_	1}.
\end{equation}
The matrix $E$, defined in \eqref{eq:diag} and modified in \eqref{eq:dec} and \eqref{eq:I}, is essentially a sum of several products of Gaussian matrices. Using Lemma \ref{lm:1} and Lemma \ref{lm:3}, we obtain (see Proposition 2 below) the following bound:
\[\|E\|_2\leq 2 E_3(\epsilon,\Sigma_1,k,n,\gamma),\]
holds with probability over $1-4e^{-(n-k)\gamma^2/2}$
Here,
\begin{eqnarray}\label{eq:E3} E_3(\epsilon,\Sigma,k,n,\gamma) &= &\epsilon^3 \|\Sigma_1^{-1}\|_2^2(\alpha_1^2\alpha_3+2\alpha_1^2\alpha_2+2\alpha_1\alpha_2\alpha_3)+\epsilon^4 \|\Sigma_1^{-1}\|_2^3(\alpha_1^2\alpha_2^2+2\alpha_1^2\alpha_2\alpha_3)\notag \\ &+&\epsilon^5 \|\Sigma_1^{-1}\|_2^4\alpha_1^2\alpha_2^2\alpha_3 ,\end{eqnarray}
where $\alpha_1-\alpha_3$ are the same as those defined in Theorem \ref{thm:main}. \\
The right hand side of \eqref{eq:lemma1} therefore has the following bound:
\begin{equation}\label{eq:num}
\sqrt{\|E^T\widetilde{U}_1\|_{F}^{2}+\|E\widetilde{V}_1\|^{2}_{F}}\leq \sqrt{2k}\|E\|_2\leq 2\sqrt{2k}E_3.
\end{equation}
 We are now ready to define the rotation $M$ which first appears in \eqref{eq:decom1}. Let\begin{equation}\label{eq:M}
M:=(M_1\widehat{L})^{-1},
\end{equation}
where $\widehat{L}$ is the minimizer of \eqref{eq:lemma1} and by Theorem 1, $\widehat{L}=Z_1Z_2^T$, with $Z_1SZ_2^T$ being any SVD of $\widehat{U}_1^T\widetilde{U}_1+\widehat{V}_1^T\widetilde{V}_1$. Plugging \eqref{eq:num} and \eqref{eq:M} into \eqref{eq:lemma1} ,we obtain:
 \[
 \|\widetilde{U}_1M-Q_{(UP)_1}\|_{\max}\leq \|\widetilde{U}_1M-Q_{(UP)_1}\|_F= \|\widetilde{U}_1-Q_{(UP)_1}M^{-1}\|_F =  \|\widetilde{U}_1-Q_{(UP)_1}M_1\widehat{L}\|_F\leq \frac{4\sqrt{k}E_3}{\delta_1}.
 \]

\subsection{ESTIMATING III}
We start with breaking $III$ into two parts:
\begin{eqnarray*}
III&=&\epsilon^2\|U_2B\|_{\max}\notag\\&\leq& \epsilon^2(\|U_2C_{22}C_{12}^T\Sigma_{1}^{-2}\|_{\max}+\|U_2C_{21}\Sigma_{1}^{-1}C_{11}\Sigma_1^{-1}\|_{\max})\notag\\&=&\epsilon^2(IV + V).
\end{eqnarray*}
We estimate IV and V separately.
\begin{eqnarray*}\label{eq:IV}
IV&=&\left\|\left(U_1,U_2\right)\left(\begin{matrix}C_{21}^T \\ C_{22}\end{matrix}\right)C_{12}^T\Sigma_1^{-2}-U_1C_{21}^TC_{12}^T\Sigma_1^{-2}\right\|_{\max} \notag\\
&\leq& \|\Sigma_1^{-1}\|_2^2\left(\left\|\left(U_1,U_2\right)\left(\begin{matrix}C_{21}^T \\ C_{22}\end{matrix}\right)C_{12}^T\right\|_{\max}+\|U_1C_{21}^TC_{12}^T\|_{\max} \right)\notag\\
&\leq& \|\Sigma_1^{-1}\|_2^2\left(\left\|\left(U_1,U_2\right)\left(\begin{matrix}C_{21}^T \\ C_{22}\end{matrix}\right)C_{12}^T\right\|_{\max}+k\|U_1\|_{\max}\|C_{21}^TC_{12}^T\|_{\max}\right).
\end{eqnarray*}
Observe that the entries of $(U_1,U_2)\left(\begin{matrix}C_{21}^T \\ C_{22}\end{matrix}\right)$ are i.i.d. $N(0,1/n)$ and are independent of those in $C_{12}^T$. Therefore we can apply Lemma 2 to each entry $\left(U_1,U_2\right)\left(\begin{matrix}C_{21}^T \\ C_{22}\end{matrix}\right)C_{12}^T$ and those of $C_{21}^TC_{12}^T$ to get, with probability at least $1-2e^{-(n-k)^{\beta}+\ln(k(n+k))}$, with $0<\beta<\frac{1}{2}$,
\[
\left\|\left(U_1,U_2\right)\left(\begin{matrix}C_{21}^T \\ C_{22}\end{matrix}\right)C_{12}^T\right\|_{\max}\leq (n-k)^{-\frac{1}{2}+\beta}, \ \ \  \|C_{21}^TC_{12}^T\|_{\max}\leq (n-k)^{-\frac{1}{2}+\beta}.
\]
Therefore,
\[
IV \leq (1+k)(n-k)^{-\frac{1}{2}+\beta}\|\Sigma_1^{-1}\|_2^2.
\]
For $V$, we first observe the following upper bound,
\begin{equation}\label{eq:V} V\leq k\|\Sigma_1^{-1}\|^2_2\|U_2C_{21}\|_{\max}\|C_{11}\|_{\max}.\end{equation}
Using \eqref{eq:V}, the union bound, as well as the following inequality,
\[\frac{2}{\sqrt{\pi}}\int\limits_{x}^{\infty}e^{-t^2}dt\leq \frac {e^{-x^2}}{\sqrt{\pi}x}.
\]
We can estimate the probability that $V$ exceeds the value $2kn^{-1+\beta}\|\Sigma_1^{-1}\|_2^2$,
\begin{eqnarray}&&P(V\geq 2kn^{-1+\beta}\|\Sigma_1^{-1}\|_2^2) \notag\\
&\leq& P(k\|\Sigma_1^{-1}\|^2_2\|U_2C_{21}\|_{\max}\|C_{11}\|_{\max}\geq 2kn^{-1+\beta}\|\Sigma_1^{-1}\|_2^2) \notag \\
&=&P(\|U_2C_{21}\|_{\max}\|C_{11}\|_{\max}\geq 2n^{-1+\beta})\notag\\ &\leq& P(\|(U_2C_{21})\|_{\max}\geq \sqrt{2}n^{\frac{-1+\beta}{2}})+P(\|C_{11}\|_{\max}\geq \sqrt{2}n^{\frac{-1+\beta}{2}}) \notag\\ &\leq&knP((U_2C_{21})(1,1)\geq \sqrt{2}n^{\frac{-1+\beta}{2}})+k^2P(C_{11}(1,1)\geq \sqrt{2}n^{\frac{-1+\beta}{2}})\notag \\
&<&k(n+k)\sqrt{\frac{2}{\pi}}\int\limits_{\sqrt{2}n^{\beta/2}}^{\infty}e^{-\frac{t^2}{2}} dt \notag \\ &\leq& e^{-n^{\beta}+\ln k(n+k)},\notag
\end{eqnarray}
where $((U_2C_{21})(1,1)$ and $C_{11}(1,1)$ denotes the first elements of $U_2C_{21}$ and $C_{11}$, respectively.
Therefore, with probability exceeding $1-e^{-n^{\beta}++\ln k(n+k)}$, we have
\begin{eqnarray*}
V &\leq& 2kn^{-1+\beta}\|\Sigma_1^{-1}\|_2^2.
\end{eqnarray*}
We combine the estimates of $IV$ and $V$ to get, with probability greater than $1-3e^{-(n-k)^{\beta}+\ln(k(n+k))}$,
\begin{equation}\label{eq:E4}
III \leq \epsilon^2(1+k)(1+\frac{2}{n^{1/2}})(n-k)^{-\frac{1}{2}+\beta}\|\Sigma_1^{-1}\|_2^2  \equiv E_4.
\end{equation}
We now aggregate the estimates of $I$, $II$ and $III$ and add up all the probabilities of failure to get \eqref{eq:thm}.
\end{proof}
\begin{Proposition}
Let $U$, $P$, $\Sigma_1$, $n$, $k$ and $\epsilon$ be the same as defined the proof of the Theorem \ref{thm:main}. Then for any $\gamma>0$, it holds with probability exceeding $1-4e^{-(n-k)\gamma^2/2}$ that
\[ \|(UP)_1^T(UP)_1-I\|_2\leq E_1,\]
where $E_1=\epsilon^2\|\Sigma_1^{-1}\|_2^2\alpha_1^2+2\epsilon^3\|\Sigma_1^{-1}\|_2^3\alpha_1^2\alpha_2+\epsilon^4\|\Sigma_1^{-1}\|^4_2\alpha_1^2\alpha_2^2,$ with $\alpha_1=1+\gamma+\sqrt{\frac{k}{n}}$, $\alpha_2=3+2\gamma+2\sqrt{\frac{k}{n}}$, and $\alpha_3=2+\gamma$.
\end{Proposition}
\begin{proof}

Let
\begin{eqnarray*}
F&=&\left(\begin{matrix} 0 & -\Sigma_1^{-1}C_{21}^{T} \\ C_{21}\Sigma_1^{-1} & 0 \end{matrix}\right),\\ G&=&\left(\begin{matrix} 0 & B^T \\ -B & 0 \end{matrix}\right),\\ H&=&\left(\begin{matrix} 0 & -\Sigma_1^{-1}C_{12} \\ C_{12}^T\Sigma_1^{-1} & 0 \end{matrix}\right), \\J&=&\left(\begin{matrix} 0 & D \\ -D^T & 0 \end{matrix}\right).
\end{eqnarray*}
By the definition of $P$ and direct calculations, we get
\begin{equation}\label{eq:expan}
(UP)^T(UP)-I=P^TP-I=\epsilon^2 F^TF+\epsilon^3 (F^TG+G^TF)+\epsilon^4 G^TG.
\end{equation}
Applying Lemma \ref{lm:3} to $C$,  we obtain that with probability exceeding $1-4e^{-(n-k)\gamma^2/2}$,
\[
\|C_{12}\|_2,\|C_{21}\|_2\leq \alpha_1, \ \ \ \|C_{11}\|_2+\|C_{12}\|_2\leq \alpha_2, \ \ \ \|C\|_2, \|C_{22}\|_2\leq \alpha_3.
\]
When these bounds holds, we can get a bound on the matrix $B$ as follows
\begin{eqnarray*}
\|B\|&=&\|C_{22}C_{12}^T\Sigma_1^{-2}+C_{21}\Sigma_1^{-1}C_{11}\Sigma_1^{-1}\|_2\\
&\leq& (\|C_{22}\|_2\|C_{12}\|_2+\|C_{21}\|_2\|C_{11}\|_2)\|\Sigma_1^{-1}\|_2^2\\
&\leq & \alpha_1\alpha_2\|\Sigma_1^{-1}\|_2^2.
\end{eqnarray*}
Similarly, we have
\begin{equation}\label{eq:ineq}\|D\|\leq \alpha_1\alpha_2\|\Sigma_1^{-1}\|_2^2,\ \ \|G\|_2,\|J\|_2\leq \alpha_1\alpha_2\|\Sigma_1^{-1}\|_2^2, \ \ \ \|H\|_2,\|F\|_2\leq \alpha_1\|\Sigma_1^{-1}\|_2.\end{equation}
Inserting these inequalities into \eqref{eq:expan} finishes the proof.
\end{proof}
\begin{Proposition}
Let $U$, $P$, $\Sigma_1$, $n$, $k$ and $\epsilon$ be the same as defined the proof of the Theorem 5. In addition, assume that $\epsilon$ is small enough such that the $E_1$ defined in Proposition 1 satisfies $E_1\leq \frac{1}{2\sqrt{k}}$. Then for any $\gamma>0$, it holds with probability exceeding $1-4e^{-(n-k)\gamma^2/2}$ that
\[ \|E\|_2\leq 2 E_3,\]
where
\begin{eqnarray*} E_3 &= &\epsilon^3 \|\Sigma_1^{-1}\|_2^2(\alpha_1^2\alpha_3+2\alpha_1^2\alpha_2+2\alpha_1\alpha_2\alpha_3)+\epsilon^4 \|\Sigma_1^{-1}\|_2^3(\alpha_1^2\alpha_2^2+2\alpha_1^2\alpha_2\alpha_3)\notag \\ &+&\epsilon^5 \|\Sigma_1^{-1}\|_2^4\alpha_1^2\alpha_2^2\alpha_3 ,\end{eqnarray*}  with $\alpha_1$-$\alpha_3$ being the same as in Proposition 1.
\end{Proposition}
\begin{proof}
The expression of the error matrix $E$ when it was first defined in \eqref{eq:diag} is
\begin{eqnarray}\label{eq:expa}
E&=&\epsilon^3(CJ+G^TC+F^TCH+F^TJ+G^TH)+\epsilon^4(G^TJ+F^TCJ+G^TCH)\notag\\&+&\epsilon^5G^TCJ.
\end{eqnarray}

Using \eqref{eq:ineq} on \eqref{eq:expa} to derive
\[\|E\|_2\leq E_3.\]
Recall that $E$ is changed in \eqref{eq:dec} by multiplying $\left( \begin{matrix}
R_{(UP)_1}^{-T} & 0 \\ 0 & R_{(UP)_2}^{-T}
\end{matrix}\right)$ on the left and $\left( \begin{matrix}
R_{(VO)_1}^{-1} & 0 \\ 0 & R_{(VO)_2}^{-1}
\end{matrix}\right)$ on the right. Due to \eqref{eq:inverse}
\[
\|R_{(UP)_1}^{-1}\|_2=\frac{1}{\sigma_{min}((UP)_1)}<\sqrt{2},
\]
so are $R_{(UP)_2}^{-1}$ and $R_{VO)_i}^{-1}$ with $i=1,2$. Therefore after the change in \eqref{eq:dec},
\[
\|E\|_2\leq 2 E_3.
\]
At last, the change of $E$ made in \eqref{eq:I} does not change the value of $\|E\|_2$. Here completes the proof.
\end{proof}
\begin{remark}\label{rmk:4}
By including higher order adjudgements of $\epsilon$ ($\epsilon^l, l=3,4,...$) into the definition of $P$ and $O$ in \eqref{eq:OP}, we can make the error term $E_3$ smaller, i.e., $E_3= O(\epsilon^l)(l=4,5,...)$ , while keeping the order of $E_1,E_2, E_4$ and $\delta_1$ unchanged. As a consequence, a better estimate will be derived. However, in doing so, the calculation will be much more complicated, which is why we choose to present the near optimal result with a simpler proof.  
\end{remark}

\section{Application}
In this section, we show how can Theorem \ref{thm:main} be used on the M-PSK (Phase Shift Keying) classification problem.
\subsection{The MPSK classification problem}
PSK is a modulation scheme which uses the phases of sinusoids to encode digital data. If the total number of phases in use is $M$, the modulation is called M-PSK. Special names are given to the two most popular PSK modulation types: the 2PSK and the 4PSK, they are often called BPSK and QPSK, respectively. Other useful PSK types include 8PSK, 16PSK and 32PSK.

The MPSK classification problem is to determine the number of phases in an incoming M-PSK signal.

Figure \ref{fig:BPSK} plots a BPSK signal, where the phase of the sinusoid is reassigned in every two seconds. The continuous parts between every two consecutive reassignments are called symbols, and the two seconds is called symbol period or symbol duration. BPSK has two phases of choice, $0$ and $\pi$, so it has two type of symbols $\cos t$ and $-\cos t$.(similarly an M-PSK modulation has $M$ symbols of choice). If we use the phase $\pi$ to encode the binary value 1 and phase $0$ to encode 0, then the digital signal $00111101$ after modulation becomes the analog signal in Figure \ref{fig:BPSK}.  \\
The mathematical description of a noisy MPSK signal $s(t)$ is as follows:
\begin{equation}\label{eq:mpsk}
s(t)=\sum\limits_{n\in \mathbb{Z}} \chi_T(t-nT)\cos(2\pi f_c t+\theta_n+\theta_c)+w(t),
\end{equation}
where for any $a>0$, $\chi_a$ denotes the characteristic function of $[0,a]$. $T$ denotes the \emph{symbol period/duration}; The sinusoidal function $cosine$ is called the \emph{carrier} wave, its frequency $f_c$ is called the \emph{carrier frequency}, and its initial phase $\theta_c$ is called the\emph{carrier phase}. $w(t)$ represents a noise term, it is usually assumed to be additive white Gaussian with two sided power spectral density $N_0/2$, for some positive number $N_0$. The digital information is encoded in the phase parameters $\theta_n \in \Theta_M=\{\frac{2\pi i}{M}, i=0,1,...,M-1\}$, where $\theta_n$ denotes the phase of the $n$'th symbol. It is easy to see that if $M=2^m$, then each symbol can encode $m$ binary bits. We assume that all the parameters are fixed for the duration of the signal.

An important concept related to the PSK modulation is the so-called constellation diagram. It is a two-dimensional scattered plot of all the available phases in a modulation scheme. Specifically, the constellation diagram for MSPK is the graph of the points: $\left\{\left(\cos \frac{2k\pi}{M},\sin \frac{2k\pi}{M}\right),\ \  k=1,...,M\right\}$ (see Figure \ref{fig:constel} ). Since constellation diagrams are one to one to the modulation types, recovering the diagram is equivalent to classifying MPSK signals.
\subsection{Model setting and the proposed method}
Different applications may result in different model settings about which parameters are known or unknown in \eqref{eq:mpsk}. In particular, the easiest setting is to assume $M$ to be the only unknown parameter, and the hardest is the fully blind classification, which assumes none of the parameters is known. We refer the readers to \cite{Polydoros}, \cite{Soliman}, \cite{Swami} for some classical methods and settings.


For illustration purposes, we consider a simple, yet nontrivial, partially blind model assuming that $T$ and $N_0$ are known, $\theta_n$ is chosen uniformly from $\Theta_M$ (which is a common assumption on the MPSK signals), $f_c=k/T$ for some unknown integer $k$, and $M$, $\theta_c$ are unknown.

Suppose we sample $s(t)$ in \eqref{eq:mpsk} in the following way. We take $L$ uniform samples from each symbol for a period of $N$ symbols and store them as an $L\times N$ data matrix $Y$, whose element $Y(l,n)$ denotes the $l$th sample in the $n$th period, which hence has the expression:
\begin{eqnarray}
Y(l,n)&=&s\left((L(n-1)+l)\frac{T}{L}\right) \notag \\ &=&
\cos\left(2\pi f_c (L(n-1)+l)\frac{T}{L}+\theta_n+\theta_c\right)+W(l,n). \notag \\
&=& cos\left(2\pi f_c \frac{lT}{L}+\theta_n+\theta_c\right)+W(l,n),
\end{eqnarray}
where $W$ is a Gaussian noise matrix with i.i.d. entries. For the illustration purpose and WLOG, we assume the entries of $W$ obey $N(0,1)$ distribution.

When $W=0$, since there are only $M$ values for $\theta_n$, $n=1,...,N$, the columns of $Y$ also have $M$ patterns. From the dimension reduction point of view, columns of $Y$ are merely high dimensional representations of a zero dimensional manifold. When noise is added, these $M$ points become $M$ clusters. If we can determine the number of clusters by some clustering algorithm, then our classification problem is solved.

The complexity of nearly all well-known clustering methods, such as k-means and Mean Shift methods grow exponentially with dimensionality. Therefore, for large data sets, we propose to conduct dimension reductions before doing the clustering. Comparing to other methods (\cite{Polydoros}, \cite{Soliman}, \cite{Swami}) which primarily requires an additional carrier removal step, the dimension reduction based method has two advantages:
\begin{itemize}
\item  the sampling rate is allowed to be lower than the Nyquist rate of carrier wave.
\item  classification and detection are completed simultaneously.
\end{itemize}
Here, we choose SVD for dimension reduction due to the key observation that the 2 dimensional embedding of $Y$ obtained by SVD is a good approximation to the constellation diagram of the original signal $s(t)$ (see Figure \ref{fig:cluster}).
Once the approximated constellation diagram is plotted, clustering algorithms can be applied to find the exact number of clusters.
\subsection{Method validation}
Now we provide more details to legislate our method.
\begin{figure} \begin{center}
\includegraphics[width=90mm, height=70mm] {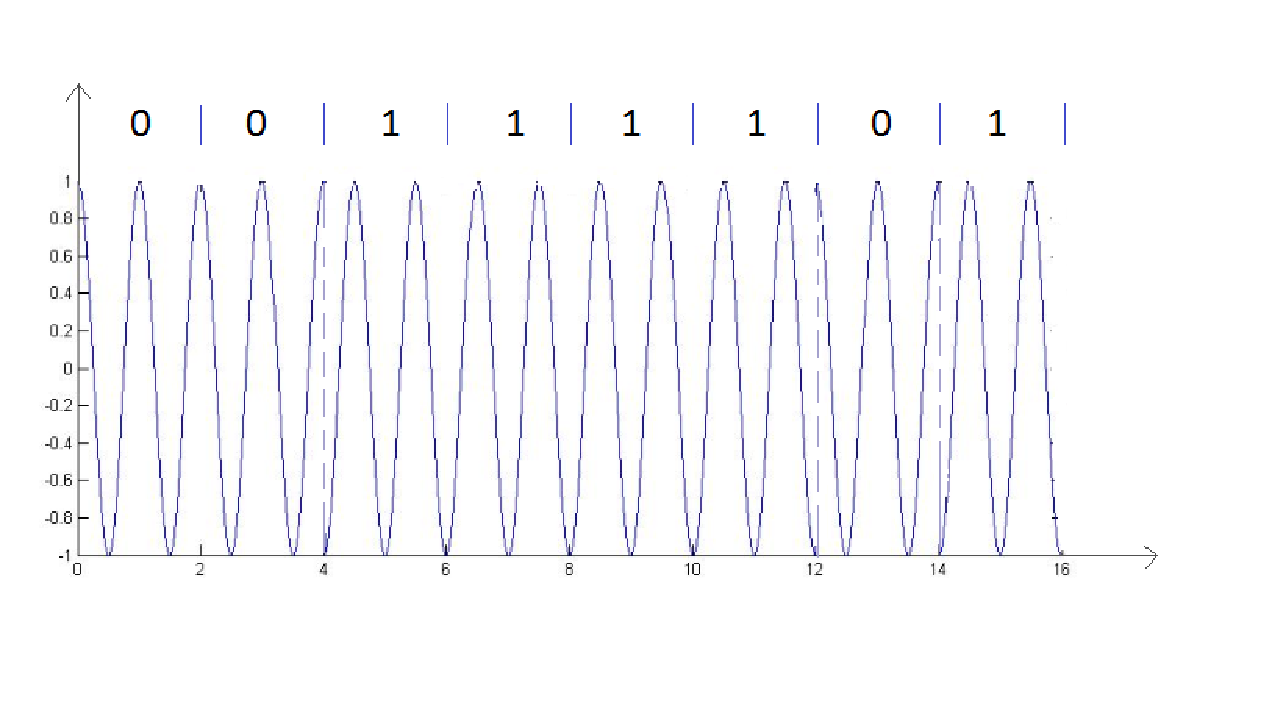}
\caption{Encoding using BPSK modulation}
\label{fig:BPSK}
\end{center}
\end{figure}
\begin{figure} \begin{center}
\includegraphics[width=100mm, height=70mm] {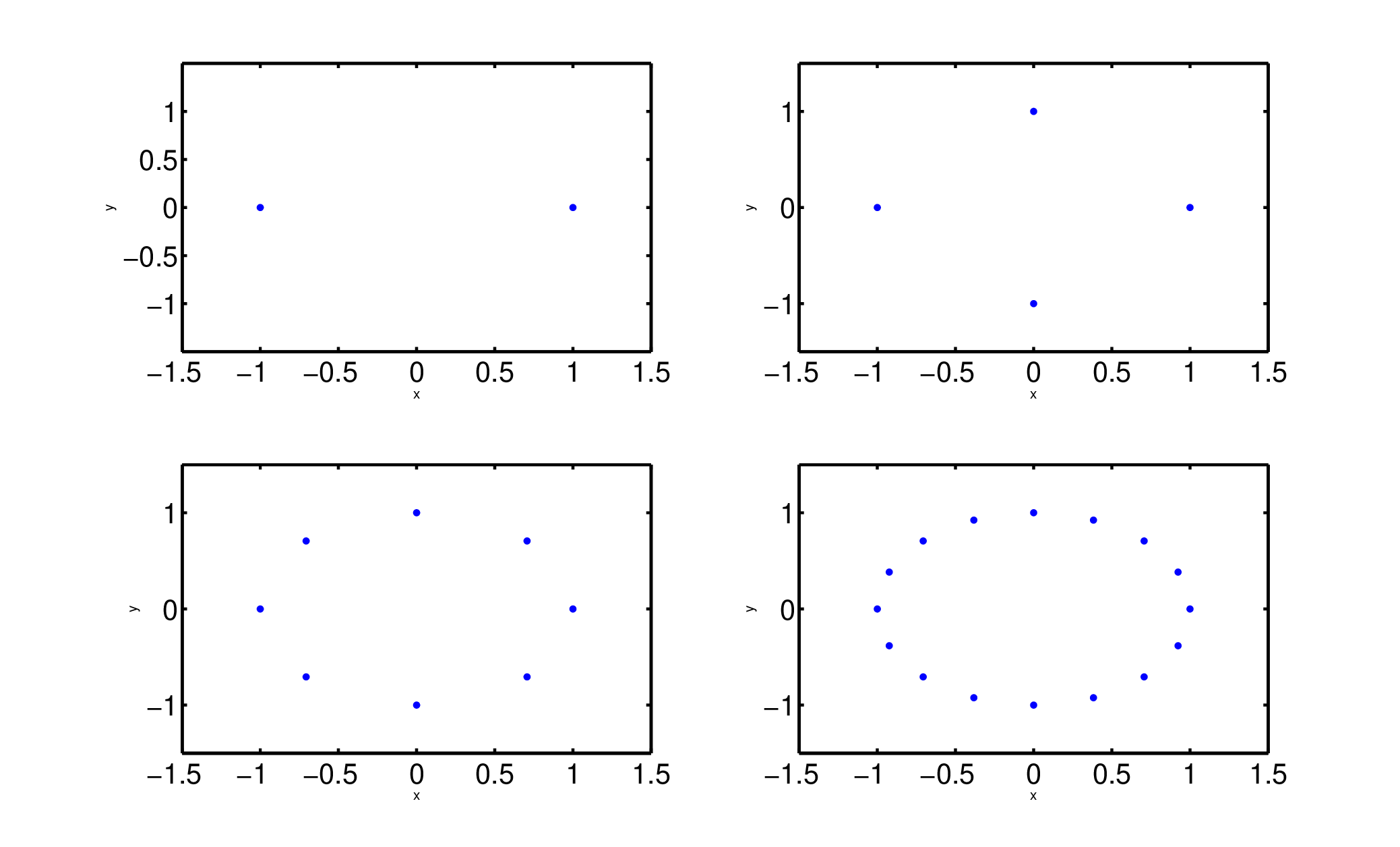}
\caption{Constellation of BPSK, QPSK, 8PSK and 16PSK modulations}
\label{fig:constel}
\end{center}
\end{figure}

By definition, $Y$ can be decomposed (not the usual SVD) as follows:
\begin{eqnarray*}
Y&=&U\Sigma V^T+W,
\end{eqnarray*}
where
\begin{eqnarray}\label{eq:decomposition}
U_{l,j}&=&\sqrt{\frac{2}{L}} cos(2\pi f_c\frac{lT}{L}+\theta_c-(j-1)\frac{\pi}{2}), \notag \\
V_{n,j}&=&\sqrt{\frac{2}{N}}cos(\theta_n+(j-1)\frac{\pi}{2}),\notag \\
\Sigma&=&\left(\begin{matrix}\frac{\sqrt{LN}}{2} &0 \\ 0 &\frac{\sqrt{LN}}{2}\end{matrix}\right),
\end{eqnarray}
for $l=1,..,L$, $j=1,2$, and $n=1,..,N$.

It can be shown that when $L \nmid f_cT$ ($L$ does not divide $f_cT$), $U$ and $V$ satisfy $U^TU=I$ and $V^TV\approx I$ with high probability. The orthogonality of $U$ can be immediately verified from its expression, but proving the nearly orthogonality of $V$ is more delicate. Roughly speaking, it follows from the uniformity assumption of $\theta_n$ and the central limit theorem. To avoid distraction from our main purpose, we put the rigorous proof in \cite{me}.

Since both $U$ and $V$ are (nearly) orthogonal, then $U\Sigma V^T$ is close to the actual SVD of $Y$ when $W=0$ or small. In other words, the right singular vectors of $Y$ are good approximations of $V$, whose columns are exactly the constellation points $(\cos(\theta_n),\sin(\theta_n))$. Figure \ref{fig:cluster} plots the right singular vector of an instance of $Y$, and it is indeed close to the constellation diagram.\\
\begin{figure} \begin{center}
\includegraphics[width=100mm, height=70mm] {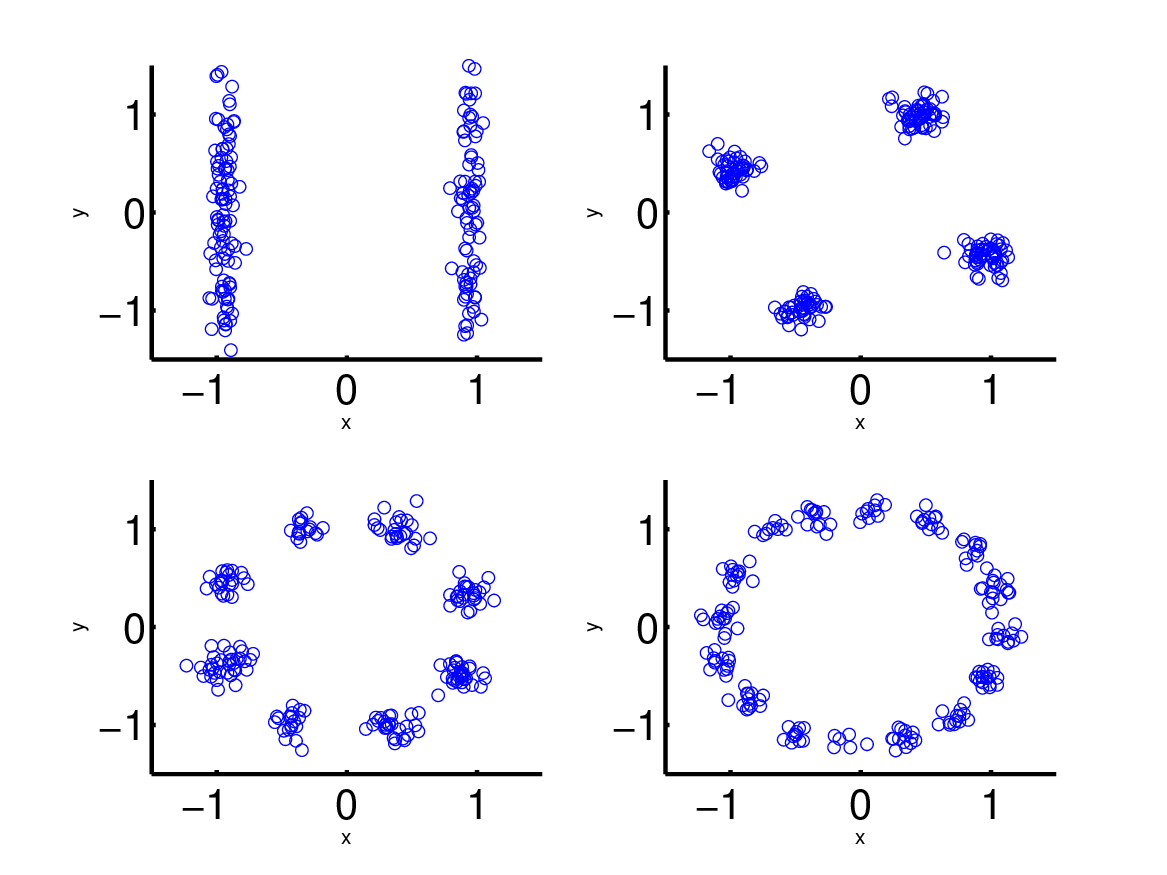}
\caption{Dimension reduction results for BPSK, QPSK, 8PSK, 16PSK modulations, SNR=14, numbers of samples=4600}
\label{fig:cluster}
\end{center}
\end{figure}
It might be easy for a human observer to tell how many clusters there are in Figure \ref{fig:cluster} without any prior knowledge, but most clustering algorithms require additional information such as  the number of clusters or the cluster radius as input. While both parameters are not easy to obtain in this model, many previous work suggest to do brute force on all the possible values of $M$ (that are 2, 4, 8, 16, 32, maybe also 64, 128), and compare the clustering results in some ad hoc way to decide which fits the data best.

Our theoretical result avoids the brute force by providing a way to approximate the cluster radius.

In order to be able to apply our result in Theorem \ref{thm:main}, we define $\widetilde{Y}$ by
\begin{equation}\label{eq:norm}
\widetilde{Y}\equiv 2Y/\sqrt{LN}=UV^T+\frac{2}{\sqrt{LN}}W=UV^T+\frac{2}{\sqrt{L}}\cdot \frac{1}{\sqrt{N}}W,
\end{equation}
so that the singular values of $\widetilde{Y}$ no longer change with the sample size.\\
In hardware implementations, larger values of $L$ are usually harder to realize than those of $N$, because $L$ corresponds to the sampling rate and $N$ to the sampling duration. Hence, in what follows, we assume that $L\leq N$. Padding zeros to $\widetilde{Y}$ to form an $N\times N$ matrix, Theorem \ref{thm:main} can then be applied. Since the factor $\frac{1}{\sqrt{N}}W$ in the last term of \eqref{eq:norm} is a normalized Gaussian matrix, the other factor, $\frac{2}{\sqrt{L}}$, then becomes the energy of noise, corresponding to the $\epsilon$ in Theorem \ref{thm:main} (i.e., $\epsilon=\frac{2}{\sqrt{L}}$). Only when $L\rightarrow \infty$, we have $\epsilon \rightarrow 0$. In other words, if we want the noise to go to 0, we must let the matrix size tend to infinity. Remark \ref{rmk:2} indicates that in order for the noise to stay Gaussian after SVD, not only do we need $\epsilon$ and $\frac{1}{\sqrt N}$ go to 0, but these two quantities must satisfy certain relationship.
Specifically, \eqref{eq:epsilon} of Remark \ref{rmk:2} implies that we must require $\epsilon$ to satisfy
\begin{equation}\label{eq:epsilon2} \epsilon =o\left( \min\left\{N^{-\beta}, N^{-1/4}, \frac{1}{\|V\|_{\max}N^{\frac{1}{2}}}\right\}\right). \end{equation}
In order to simplify \eqref{eq:epsilon2}, first observe that in this example, we have $\|V\|_{\max}=\sqrt{\frac{2}{N}}$ from \eqref{eq:decomposition}. Second, to maintain a fixed failure rate in Theorem \ref{thm:main}, say $1-\rho$ with $0<\rho<1$, it is sufficient to choose $\beta$ such that $(N-k)^{-\beta}\leq (\ln N(k+1)-\ln \rho/4)^{-1}$. These observations indicate that $N^{-1/4}$ is asymptotically the smallest term on the right hand side of \eqref{eq:epsilon2} . Hence \eqref{eq:epsilon2} is reduced to
\[
\epsilon=o( N^{-1/4} ).
\]
Since $\epsilon=\frac{2}{\sqrt{L}}$, the above relation requires $ N=o(L^2)$. This relation does not contradict with the previous assumption $L\leq N$, so the feasible region is non empty. By Remark 4, we know that the requirement on $L$ could be further relaxed to $L=O(1)$. 

Now for feasible $L$ and $N$, we can safely assume that the right singular vectors of $Y$ is a rotation of $V$ plus a random Gaussian noise. Because of the Gaussian, it is suitable to apply the Mean Shift (MS) clustering method with Gaussian kernel. The explicit form of the Gaussian matrix is given in \eqref{eq:thm}. From it, one can derive that the $95\%$ percentile of the 2D Gaussian noise on each data point is approximately $2.45\sqrt{2N_0(1-2/N)/(LN)}$. We set this number to be the radius of all clusters and feed it into the MS algorithm.

When the modulation type is BPSK, the rank of $Y$ is one, so the second dimension of its singular vectors is no longer reliable (see the up left graph of Figure \ref{fig:cluster}). Fortunately, this case can be easily detected by simply examining whether the second singular value of $Y$ is much smaller than the first one.

In our first experiment, we generate a QPSK signal with carrier frequency of 1GHZ, symbol rate 10MHZ, and damped by AWGN with $SNR=10$. We take 31 samples per symbol (much lower than the Nyquist rate of the carrier frequency, and satisfies $L\nmid f_cT$) and sample 200 symbols. In Figure \ref{fig:radius}, the two dimensional embedding of $Y$ obtained by SVD is plotted, together with a circle whose radius is the theoretical prediction $2.45\sqrt{2N_0(1-2/N)/(LN)}$. We can see that the predicted radius is very close to the real ones.
\begin{figure} \begin{center}
\includegraphics[width=100mm, height=70mm] {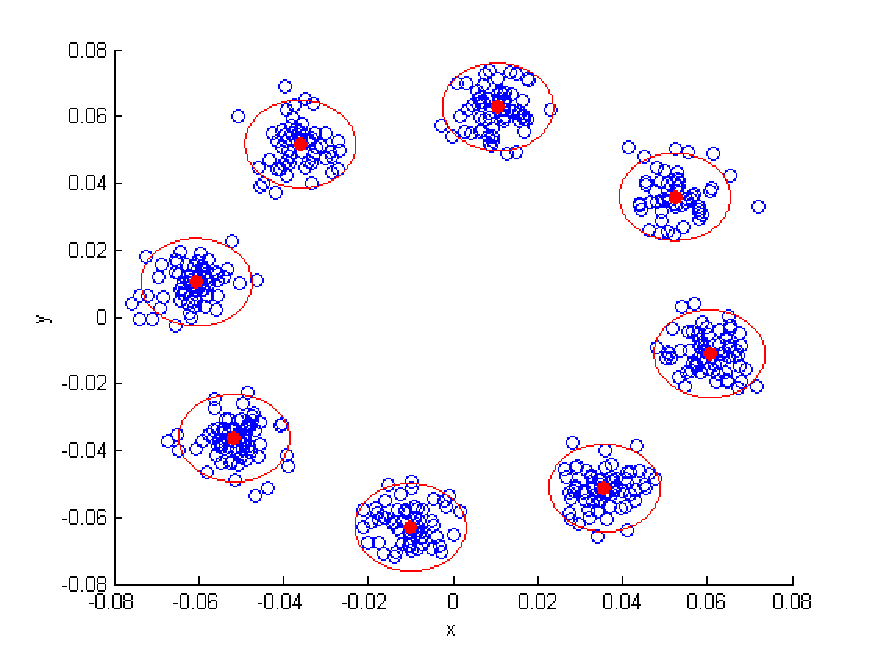}
\caption{Clustering result: Small circles denote the low dimension representation of the data points returned by PCA. The eight big circles denote the theoretically predicted cluster radius. }
\label{fig:radius}
\end{center}
\end{figure}

In our second experiment, we let the SNR decrease and examine the performance of the above algorithm. A classification is deemed as successful only when the number of clusters returned by the MS algorithm $\widehat{M}$ is strictly equal to the true $M$. The result is plotted in Figure \ref{fig:success}. As expected, when noises grow, the singular vector distributions deviate from Gaussian and the predicted radii become too small for the algorithm to find the correct $M$.

\begin{figure} \begin{center}
\includegraphics[width=100mm, height=70mm] {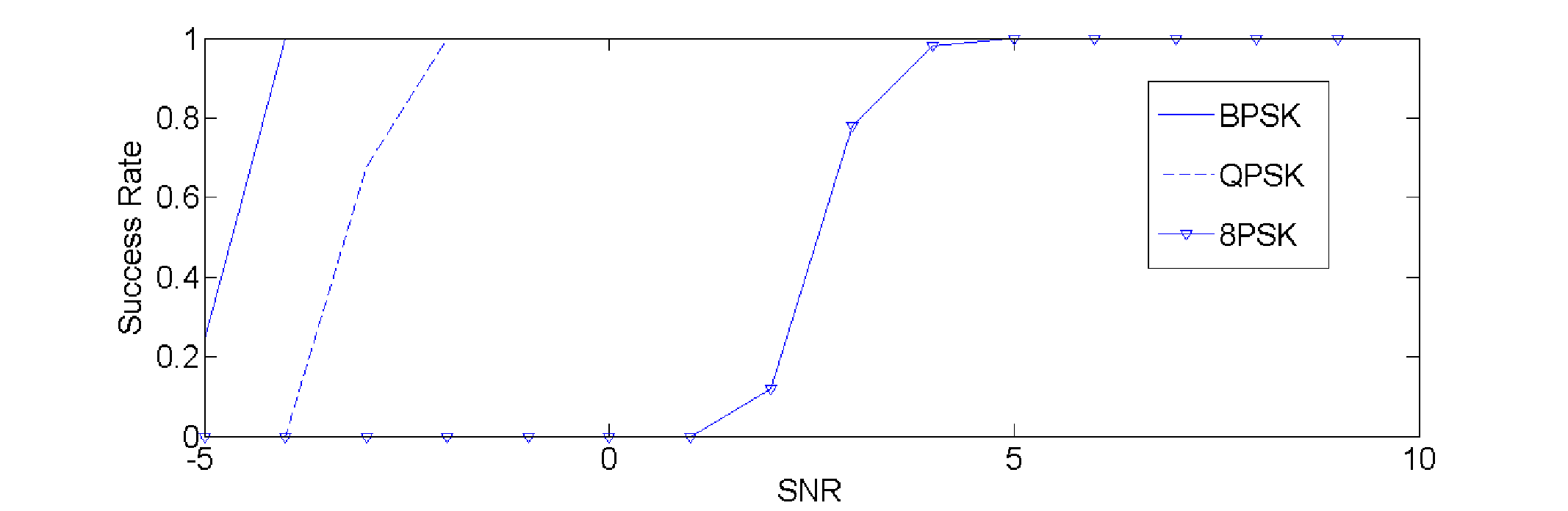}
\caption{The success rate for BPSK, QPSK, and 8PSK modulations with respect the SNR }
\label{fig:success}
\end{center}
\end{figure}

\section{CONCLUSION}
In this note, we provided a condition under which the perturbation of the principal singular vectors of a matrix under Gaussian noise has a near-Gaussian distribution. The condition is non asymptotic and is useful in application. We provided a simple example of audio signal classification problem to illustrate how our theorem can be used to make sampling strategy and to form new classification technique. More details about this new classification scheme is discussed in \cite{me}.
\section*{ACKNOWLEDGEMENT}
The author deeply grateful to Wojciech Czaja, Xuemei Chen, Ernie Esser, and Dane Taylor  for their generous help on this article. Research presented in this note was supported in part by Laboratory of Telecommunication Sciences. We gratefully acknowledge this support.

\end{document}